\documentclass[12pt,reqno]{amsart}
\usepackage{amsmath, amsthm, amssymb}

\advance \topmargin by -\headheight
\advance \topmargin by -\headsep
     
\evensidemargin \oddsidemargin
\newtheorem{theorem}{Theorem}[section]
\newtheorem{lemma}[theorem]{Lemma}
\newtheorem{corollary}[theorem]{Corollary}
\newtheorem{conjecture}[theorem]{Conjecture}

\theoremstyle{definition}

\theoremstyle{remark}

\numberwithin{equation}{section}

\newcommand{\mmod}[1]{\,\,(\text{mod}\,\,#1)}

\def\bfc{{\mathbf c}}
\def\bfd{{\mathbf d}}

\def\bfh{{\mathbf h}}

\def\bfn{{\mathbf n}}

\def\bfz{{\mathbf z}}

\def\calA{{\mathcal A}} \def\calAbar{{\overline \calA}} 

\def\calB{{\mathcal B}} 
\def\calC{{\mathcal C}} \def\calCbar{{\overline \calC}}

\def\calD{{\mathcal D}}
\def\calE{{\mathcal E}}

\def\calL{{\mathcal L}}
\def\calK{{\mathcal K}} 
\def\calM{{\mathcal M}}
\def\calN{{\mathcal N}}

\def\calP{{\mathcal P}}
\def\calQ{{\mathcal Q}}
\def\calR{{\mathcal R}} \def\calRbar{{\overline \calR}}
\def\calS{{\mathcal S}}
\def\calT{{\mathcal T}}

\def\calX{{\mathcal X}}
\def\calZ{{\mathcal Z}}

\def\dbN{{\mathbb N}}

\def\dbZ{{\mathbb Z}}

\def\gra{{\mathfrak a}}\def\grA{{\mathfrak A}}
\def\grb{{\mathfrak b}}\def\grB{{\mathfrak B}}
\def\grc{{\mathfrak c}}\def\grC{{\mathfrak C}}
\def\grD{{\mathfrak D}}
\def\grl{{\mathfrak l}}\def\grL{{\mathfrak L}}
\def\grN{{\mathfrak N}}
\def\grn{{\mathfrak n}}
\def\grB{{\mathfrak B}}\def\grC{{\mathfrak C}}
\def\grL{{\mathfrak L}}

\def\alp{{\alpha}} 
 
\def\del{{\delta}} \def\Del{{\Delta}}  
\def\tet{{\theta}}

\def\Ups{{\Upsilon}} \def\Upshat{{\widehat \Ups}}

\def\eps{\varepsilon}

\def\le{\leqslant} \def\ge{\geqslant}

\begin{document}
\title[Exceptional sets]{Relations between exceptional sets for additive 
problems}
\author[K. Kawada]{Koichi Kawada}
\address{KK: Department of Mathematics, Faculty of Education, Iwate University, 
Morioka, 020-8550 Japan}
\email{kawada@iwate-u.ac.jp}
\author[T. D. Wooley]{Trevor D. Wooley$^*$}
\address{TDW: School of Mathematics, 
University of Bristol, 
University Walk, Clifton, 
Bristol BS8 1TW, United Kingdom}
\email{matdw@bristol.ac.uk}
\thanks{$^*$The second author is supported by a Royal Society Wolfson Research
 Merit Award.}
\subjclass[2000]{11P05, 11P55, 11B75}
\keywords{Exceptional sets, Waring's problem, Hardy-Littlewood method}
\date{}
\begin{abstract}We describe a method for bounding the set of exceptional 
integers not represented by a given additive form in terms of the exceptional 
set corresponding to a subform. Illustrating our ideas with examples stemming 
from Waring's problem for cubes, we show, in particular, that the number of 
positive integers not exceeding $N$, that fail to have a representation as the 
sum of six cubes of natural numbers, is $O(N^{3/7})$.\end{abstract}
\maketitle

\section{Introduction} Bounds on exceptional sets in additive problems can 
oftentimes be improved by replacing a conventional application of Bessel's 
inequality with an argument based on the introduction of an exponential sum over
 the exceptional set, and a subsequent analysis of auxiliary mean values 
involving the latter generating function. Such a strategy underlies the earlier 
work concerning slim exceptional sets in Waring's problem due to one or both of 
the present authors (see \cite{slim1}, \cite{slim2}, \cite{slim3}, \cite{slim4},
 \cite{slim5}). Exponential sums over sets defining the additive problem at hand
 are intrinsic to the application of the Hardy-Littlewood (circle) method that 
underpins such approaches. One therefore expects each application of such a 
method to be highly sensitive to the specific identity of the sets in question. 
Our goal in this paper is to present an approach which, for many problems, is 
relatively robust to adjustments in the identity of the underlying sets. We 
illustrate our conclusions with some consequences for Waring's problem, paying 
attention in particular to sums of cubes.\par

In order to present our conclusions in the most general setting, we must 
introduce some notation. When $\calC \subseteq \dbN$, we write $\calCbar$ for 
the complement $\dbN \setminus \calC$ of $\calC$ within $\dbN$. When $a$ and 
$b$ are non-negative integers, it is convenient to denote by $(\calC)_a^b$ the 
set $\calC \cap (a,b]$, and by $|\calC|_a^b$ the cardinality of $\calC \cap 
(a,b]$. Next, when $\calC,\calD \subseteq \dbN$, we define
$$\calC \pm \calD=\{ c\pm d:\text{$c\in \calC$ and $d\in \calD$}\}.$$
As usual, we use $h\calD$ to denote the $h$-fold sum $\calD+\dots +\calD$. Also,
 we define $\Ups (\calC,\calD;N)$ to be the number of solutions of the equation
\begin{equation}\label{1.1}
c_1-d_1=c_2-d_2,
\end{equation}
with $c_1,c_2\in (\calC)_{2N}^{3N}$ and $d_1,d_2\in (\calD)_0^N$. The starting 
point for our analysis of exceptional sets is the inclusion
\begin{equation}\label{1.2}
\left({\overline{\calA+\calB}}-\calB\right)\cap \dbN \subseteq \calAbar.
\end{equation}
In \S2 we both justify this trivial relation, and also apply it to establish a 
relation between the cardinalities of complements of sets that encapsulates the 
key ideas of this paper. The following theorem is a special case of Theorem 
\ref{theorem2.1} below, in which we obtain a conclusion with the sets in 
question restricted to collections of residue classes.

\begin{theorem}\label{theorem1.1}
Suppose that $\calA, \calB\subseteq \dbN$. Then for each natural number $N$, 
one has
$$\left( |\calB|_0^N \left|{\overline{\calA+\calB}}\right|_{2N}^{3N}\right)^2\le 
\left|\calAbar\right|_N^{3N}\Ups ({\overline{\calA+\calB}},\calB;N).$$
\end{theorem}

The conclusion of Theorem \ref{theorem1.1} is not particularly transparent, so 
it seems appropriate to outline its significance and implications. Note first 
that $\left|\calAbar\right|_N^{3N}$ counts the number of natural numbers in the 
interval $(N,3N]$ that do not lie in $\calA$, which is to say, the {\it 
exceptional set} corresponding to $(\calA)_N^{3N}$. Likewise, we see that 
$\left|{\overline{\calA+\calB}}\right|_{2N}^{3N}$ counts the number of natural 
numbers in the interval $(2N,3N]$ that do not lie in $\calA+\calB$, and hence 
the exceptional set corresponding to $(\calA+\calB)_{2N}^{3N}$. Observe next that 
in many situations of interest, it is possible to show that the number of 
solutions $\bfc, \bfd$ of the equation (\ref{1.1}), counted by $\Ups(\calC,
\calD;N)$, is essentially dominated by the diagonal contribution with $c_1=c_2$ 
and $d_1=d_2$. Thus, under suitable circumstances, one finds that
$$\Ups({\overline{\calA+\calB}},\calB;N)\ll \left|{\overline{\calA +\calB}}
\right|_{2N}^{3N}|\calB|_0^N,$$
and then Theorem \ref{theorem1.1} delivers the bound
$$\left|{\overline{\calA+\calB}}\right|_{2N}^{3N}\ll \left|\calAbar\right|_N^{3N}/
|\calB|_0^N.$$
In this way, we are able to show that the exceptional set corresponding to 
$\calA+\calB$, in the interval $(2N,3N]$, is smaller than that corresponding to 
$\calA$, in $(N,3N]$, by a factor $O(1/|\calB|_0^N)$. With few exceptions, the 
scale of this improvement is well beyond the competence of more classical 
applications of the circle method.\par

The most immediate consequences of Theorem \ref{theorem1.1} concern additive 
problems involving squares or cubes. We begin with a cursory examination of the 
former problems in \S2. It is convenient, when $k$ is a natural number, to 
describe a subset $\calQ$ of $\dbN$ as being a {\it high-density subset of the 
$k$th powers} when (i) one has $\calQ\subseteq \{ n^k:n\in \dbN\}$, and (ii) for
 each positive number $\eps$, whenever $N$ is a natural number sufficiently 
large in terms of $\eps$, then $|\calQ|_0^N>N^{1/k-\eps}$. Also, when $\tet>0$, we 
shall refer to a set $\calR\subseteq \dbN$ as having {\it complementary density 
growth exponent smaller than $\tet$} when there exists a positive number $\del$ 
with the property that, for all sufficiently large natural numbers $N$, one has 
$\left|\calRbar\right|_0^N<N^{\tet-\del}$.

\begin{theorem}\label{theorem1.2}
Let $\calS$ be a high-density subset of the squares, and suppose that $\calA 
\subseteq \dbN$ has complementary density growth exponent smaller than $1$. 
Then, whenever $\eps>0$ and $N$ is a natural number sufficiently large in terms 
of $\eps$, one has
$$\left|{\overline{\calA+\calS}}\right|_{2N}^{3N}\ll N^{\eps -1/2}\left|\calAbar 
\right|_N^{3N}.$$
\end{theorem}

In \S2 we provide a slightly more general conclusion that captures, inter alia, 
the qualitative features of recent work on sums of four squares of primes (see 
\cite{slim1}, \cite{HK2006}). Following a consideration of some auxiliary mean 
values in \S3, we advance in \S4 to a discussion of additive problems involving 
cubes.

\begin{theorem}\label{theorem1.3}
Let $\calC$ be a high-density subset of the cubes, and suppose that $\calA 
\subseteq \dbN$ has complementary density growth exponent smaller than $\tet$, 
for some positive number $\tet$. Then, whenever $\eps>0$ and $N$ is a natural 
number sufficiently large in terms of $\eps$, one has the following estimates:
\vskip.05cm
\noindent (a)(exceptional set estimates for $\calA+\calC$)
\begin{align*}
\left|{\overline{\calA+\calC}}\right|_{2N}^{3N}&\ll N^{\eps -1/6}\left|\calAbar
\right|_N^{3N}+N^{\eps -2}\left( \left|\calAbar\right|_N^{3N}\right)^3;\\
\left|{\overline{\calA+\calC}}\right|_{2N}^{3N}&\ll N^{\eps -1/3}\left|\calAbar
\right|_N^{3N}+N^{\eps -1}\left( \left|\calAbar\right|_N^{3N}\right)^2;
\end{align*}
\vskip.05cm
\noindent (b)(exceptional set estimates for $\calA+2\calC$)
\begin{align*}
\left|{\overline{\calA+2\calC}}\right|_{2N}^{3N}&\ll N^{\eps -1/2}\left|\calAbar
\right|_N^{3N}+N^{\eps-4/3}\left( \left|\calAbar\right|_N^{3N}\right)^2,
\text{ provided that $\tet\le 1$};\\
\left|{\overline{\calA+2\calC}}\right|_{2N}^{3N}&\ll N^{\eps -2/3}\left|\calAbar
\right|_N^{3N},\text{ provided that $\tet\le {\textstyle{\frac{13}{18}}}$};
\end{align*}
\vskip.05cm
\noindent (c)(exceptional set estimates for $\calA+3\calC$)
\begin{align*}
\left|{\overline{\calA+3\calC}}\right|_{4N}^{6N}&\ll N^{\eps -5/3}\left( 
\left|\calAbar\right|_N^{6N}\right)^2,\text{ provided that $\tet \le 1$};\\
\left|{\overline{\calA+3\calC}}\right|_{4N}^{6N}&\ll N^{\eps -5/6}\left|\calAbar
\right|_N^{6N},\text{ provided that $\tet \le {\textstyle\frac{8}{9}}$}.
\end{align*}
\end{theorem}

The bounds supplied by Theorem \ref{theorem1.3}(a) have direct consequences for 
the exceptional set in Waring's problem for sums of cubes. When $s$ is a natural
 number and $N$ is positive, write $E_s(N)$ for the number of positive integers 
not exceeding $N$ that fail to be represented as the sum of $s$ positive 
integral cubes. Thus, if we define $\calC=\{ n^3:n\in \dbN\}$, then we have 
$E_s(N)=\left|{\overline{s\calC}}\right|_0^N$. In \S5 we establish the following
 estimates for $E_s(N)$. 

\begin{theorem}\label{theorem1.4}
Let $\tau$ be any positive number with $\tau^{-1}>2982+56\sqrt{2833}$. Then 
one has
$$E_4(N)\ll N^{37/42-\tau},\quad E_5(N)\ll N^{5/7-\tau},\quad E_6(N)\ll N^{3/7-2\tau}.$$
\end{theorem}

The estimate presented here for $E_4(N)$ is simply a restatement of Theorem 1.3 
of Wooley \cite{Woo2000}, itself only a modest improvement on Theorem 1 of 
Br\"udern \cite{Bru1991}. Our bound for $E_5(N)$ may be confirmed by a classical
 approach employing Bessel's inequality, and indeed such a bound is reported in 
equation (1.3) of \cite{BKW2001}. Our approach in this paper is simply to apply 
the first estimate of Theorem \ref{theorem1.3}(a). Finally, the estimate for 
$E_6(N)$ provided by Theorem \ref{theorem1.4} is new, and may be compared with 
the bound $E_6(N)\ll N^{23/42}$ reported in equation (1.3) of \cite{BKW2001}. Note
 that $\frac{23}{42}>0.5476$, whereas one may choose a permissible value of 
$\tau$ so that $\frac{3}{7}-2\tau<0.4283$. Of course, in view of Linnik's 
celebrated work \cite{Lin1942}, one has $E_s(N)\ll 1$ for $s\ge 7$.\par  

An important strength of Theorem \ref{theorem1.3} is the extent to which it is 
robust to adjustments in the set $\calC$ of cubes to which it is applied. It is 
feasible, for example, to extract estimates for exceptional sets in the 
Waring-Goldbach problem for cubes. The complications associated with inherent 
congruence conditions are easily accommodated by simple modifications of our 
basic framework. In order to illustrate such ideas, when $s$ is a natural 
number and $N$ is positive, write $\calE_6(N)$ for the number of even positive 
integers not exceeding $N$, and not congruent to $\pm 1\pmod{9}$, which fail 
to possess a representation as the sum of $6$ cubes of prime numbers. In 
addition, write $\calE_7(N)$ for the number of odd positive integers not 
exceeding $N$, and not divisible by $9$, which fail to possess a representation 
as the sum of $7$ cubes of prime numbers, and denote by $\calE_8(N)$ the number 
of even positive integers not exceeding $N$ that fail to possess a 
representation as the sum of $8$ cubes of prime numbers. A discussion of the 
necessity of the congruence conditions imposed here is provided in the preamble 
to Theorem 1.1 of \cite{slim2}. By applying a variant of Theorem 
\ref{theorem1.3}, in \S5 we obtain the following upper bounds on $\calE_s(N)$ 
$(6\le s\le 8)$.

\begin{theorem}\label{theorem1.5}
One has
$$\calE_6(N)\ll N^{23/28},\quad \calE_7(N)\ll N^{23/42}\quad \text{and}\quad 
\calE_8(N)\ll N^{3/14}.$$
\end{theorem}

For comparison, Theorem 1 of Kumchev \cite{Kum2005} supplies the weaker bounds
$$\calE_6(N)\ll N^{31/35},\quad \calE_7(N)\ll N^{17/28}\quad \text{and}\quad 
\calE_8(N)\ll N^{23/84}.$$
We have more to say concerning the Waring-Goldbach problem, so we defer further 
consideration of allied conclusions to a future occasion.\par

As the final illustration of our methods, in \S6 we consider Waring's problem 
for biquadrates. Since fourth powers are congruent to $0$ or $1$ modulo $16$, a 
sum of $s$ biquadrates must be congruent to $r$ modulo $16$, for some integer 
$r$ satisfying $0\le r\le s$. If $n$ is the sum of $s<16$ biquadrates and 
$16|n$, moreover, then $n/16$ is also the sum of $s$ biquadrates. It therefore 
makes sense, in such circumstances, to consider the representation of integers 
$n$ with $n\equiv r\pmod{16}$ for some integer $r$ with $1\le r\le s$. Define 
$Y_s(N)$ to be the number of integers $n$ not exceeding $N$ that satisfy the 
latter condition, yet cannot be written as the sum of $s$ biquadrates.

\begin{theorem}\label{theorem1.6}
Write $\del=0.00914$. Then one has
$$Y_7(N)\ll N^{15/16-\del},\quad Y_8(N)\ll N^{7/8-\del},\quad Y_9(N)\ll N^{13/16-\del},$$
$$Y_{10}(N)\ll N^{3/4-2\del},\quad Y_{11}(N)\ll N^{5/8-2\del}.$$
\end{theorem}

Here, the estimates for $Y_s(N)$ when $7\le s\le 9$ follow from a classical 
application of Bessel's inequality, combined with the work of Vaughan 
\cite{Vau1989} and Br\"udern and Wooley \cite{BW2000} concerning sums of 
biquadrates. Such techniques would also yield the bounds $Y_{10}(N)\ll 
N^{3/4-\del}$ and $Y_{11}(N)\ll N^{11/16-\del}$, each of which is inferior to the 
relevant conclusion of Theorem \ref{theorem1.6}. We remark that superior 
estimates are available if one is prepared to omit the congruence class $s$ 
modulo $16$, or $s-1$ and $s$ modulo $16$, from the integers under consideration
 for representation as the sum of $s$ biquadrates. We refer the reader to 
Theorems 1.1 and 1.2 of the authors' earlier work \cite{slim4} for details. 
Finally, we note that in view of Theorem 1.2 of Vaughan \cite{Vau1989}, one has 
$Y_s(N)\ll 1$ for $s\ge 12$.\par

In \S7 we discuss further the abstract formulation of exceptional sets 
underlying Theorem \ref{theorem1.1}, and consider the consequences of the most 
ambitious conjectures likely to hold for the additive theory of exceptional 
sets.\par

Throughout, the letter $\eps$ will denote a sufficiently small positive number. 
We use $\ll$ and $\gg$ to denote Vinogradov's well-known notation, implicit 
constants depending at most on $\eps$, unless otherwise indicated. In an effort 
to simplify our analysis, we adopt the convention that whenever $\eps$ appears 
in a statement, then we are implicitly asserting that for each $\eps>0$, the 
statement holds for sufficiently large values of the main parameter. Note that 
the ``value'' of $\eps$ may consequently change from statement to statement, 
and hence also the dependence of implicit constants on $\eps$.

\section{The basic inequality} Our goal in this section is to establish the 
upper bound presented in Theorem \ref{theorem1.1}, illustrating this relation 
with the inexpensive conclusion recorded in Theorem \ref{theorem1.2}. We begin 
by spelling out the inclusion (\ref{1.2}). The proof is by contradiction. Let 
$n\in {\overline{\calA+\calB}}$ and $b\in \calB$. Suppose, if possible, that 
$n-b\in \calA$. Then there exists an element $a$ of $\calA$ for which $n-b=a$, 
whence $n=a+b\in \calA+\calB$. But then $n\not\in {\overline{\calA+\calB}}$, 
contradicting our initial hypothesis. We are therefore forced to conclude that 
$n-b\not\in \calA$, so that if $n-b\in \dbN$, then $n-b\in \calAbar$. In this 
way, we confirm that $\left({\overline{\calA+\calB}}-\calB\right)\cap \dbN
\subseteq \calAbar$, as desired.\par

We establish Theorem \ref{theorem1.1} in a more general form useful in 
applications. In this context, when $q$ is a natural number and $\gra\in \{0,1,
\dots ,q-1\}$, we define $\calP_\gra=\calP_{\gra,q}$ by
$$\calP_{\gra,q}=\{ \gra+mq:m\in \dbZ\}.$$
Also, we describe a set $\calL$ as being a {\it union of arithmetic progressions
 modulo $q$} when, for some subset $\grL$ of $\{0,1,\dots ,q-1\}$, one has
$$\calL=\bigcup_{\grl\in \grL}\calP_{\grl,q}.$$
In such circumstances, given a subset $\calC$ of $\dbN$ and integers $a$ and 
$b$, it is convenient to write
$$\langle \calC \wedge \calL\rangle_a^b=\min_{\grl \in \grL}\left| \calC \cap 
\calP_{\grl,q}\right|_a^b.$$

\begin{theorem}\label{theorem2.1}
Suppose that $\calA,\calB \subseteq \dbN$. In addition, let $\calL$, $\calM$ and
 $\calN$ be unions of arithmetic progressions modulo $q$, for some natural 
number $q$, and suppose that $\calN\subseteq \calL+\calM$. Then for each 
natural number $N$, one has
$$\left( \langle \calB \wedge \calL\rangle_0^N \left| {\overline{\calA+\calB}}
\cap \calN\right|_{2N}^{3N}\right)^2\le q\left| \calAbar \cap \calM\right|_N^{3N}
\Ups({\overline{\calA+\calB}}\cap \calN,\calB\cap \calL;N).$$
\end{theorem} 

\begin{proof} We begin by deriving a variant of the relation (\ref{1.2}). Let 
$N$ be a large natural number, and suppose that $\calL,\calM,\calN$ satisfy the 
hypotheses of the statement of the theorem. We may suppose that there are sets 
$\grA,\grB, \grC \subseteq \{0,1,\dots ,q-1\}$ with the property that
$$\calM=\bigcup_{\gra \in \grA}\calP_\gra,\quad \calL=\bigcup_{\grb \in \grB}\calP_\grb 
\quad \text{and}\quad \calN=\bigcup_{\grc \in \grC}\calP_\grc.$$
Moreover, in view of the hypothesis $\calN \subseteq \calL+\calM$, there exists 
a subset $\grD$ of $\grB \times \grA$, with $\text{card}(\grD)\le q$, satisfying
 the property that
$$\bigcup_{\grc \in \grC}\calP_\grc =\bigcup_{(\grb,\gra)\in \grD}(\calP_\gra+\calP_\grb).$$
In particular, for each $\grc\in \grC$, there exists a pair $(\grb,\gra)\in 
\grD$ satisfying the property that $\calP_\grc =\calP_\gra+\calP_\grb$.\par

Suppose now that $\grc\in \grC$, and that $(\grb,\gra)\in \grD$ satisfies the 
condition that $\calP_\grc=\calP_\gra+\calP_\grb$. Let $N$ be a large natural 
number, and suppose that $n\in \left( {\overline{\calA+\calB}}\cap \calP_\grc
\right)_{2N}^{3N}$ and $b\in (\calB\cap \calP_\grb)_0^N$. Then
$$n-b\in (N,3N]\cap (\calP_\grc-\calP_\grb)=(N,3N]\cap \calP_\gra,$$
and so the argument in the opening paragraph of this section shows that 
$n-b\in \left(\calAbar \cap \calP_\gra\right)_N^{3N}$. We therefore deduce that
\begin{equation}\label{2.1}
\left({\overline{\calA+\calB}}\cap \calP_\grc\right)_{2N}^{3N}-(\calB \cap 
\calP_\grb)_0^N\subseteq \left(\calAbar \cap \calP_\gra \right)_N^{3N},
\end{equation}
whence
$$\left|\calAbar \cap \calP_\gra\right|_N^{3N}\ge \text{card} \left( 
\left({\overline{\calA+\calB}}\cap \calP_\grc \right)_{2N}^{3N}-(\calB \cap 
\calP_\grb)_0^N\right).$$

\par Next, write $\rho_{\grb \grc}(m)$ for the number of solutions of the equation
 $m=n-b$, with $n\in \left({\overline{\calA+\calB}}\cap \calP_\grc
\right)_{2N}^{3N}$ and $b\in (\calB\cap \calP_\grb )_0^N$. An application of 
Cauchy's inequality shows that
\begin{equation}\label{2.2}
\Bigl( \sum_{\grc \in \grC}\sum_{1\le m\le 3N}\rho_{\grb \grc}(m)\Bigr)^2\le \Bigl( 
\sum_{\grc\in \grC}\sum_{\substack{1\le m\le 3N\\ \rho_{\grb\grc}(m)\ge 1}}1\Bigr) \Bigl( 
\sum_{\grc \in \grC}\sum_{1\le m\le 3N}\rho_{\grb\grc}(m)^2\Bigr).
\end{equation}
On recalling the definition of $\Ups(\calC,\calD;N)$ from the preamble to 
Theorem \ref{theorem1.1}, we have
\begin{align*}
\sum_{\grc \in \grC}\sum_{1\le m\le 3N}\rho_{\grb\grc}(m)^2&=\sum_{\grc \in \grC}\Ups(
{\overline{\calA+\calB}}\cap \calP_\grc,\calB\cap \calP_\grb;N)\\
&\le \Ups ({\overline{\calA+\calB}}\cap \calN,\calB\cap \calL;N).
\end{align*}
Moreover, a moment's reflection confirms that
\begin{align*}
\sum_{\grc \in \grC}\sum_{1\le m\le 3N}\rho_{\grb\grc}(m)&=\sum_{\grc\in \grC}
\sum_{n\in ({\overline{\calA+\calB}}\cap \calP_\grc)_{2N}^{3N}}\sum_{b\in (\calB\cap \calP_\grb)_0^N}1\\
&\ge \min_{\grb \in \grB}\left| \calB \cap \calP_\grb\right|_0^N
\sum_{n\in ({\overline{\calA+\calB}}\cap \calN)_{2N}^{3N}}1\\
&=\langle \calB \wedge \calL\rangle_0^N\left|{\overline{\calA+\calB}}\cap \calN
\right|_{2N}^{3N}.
\end{align*}
Observe next that, in view of the relation (\ref{2.1}), when $\rho_{\grb\grc}(m)
\ge 1$, one has $m\in \left(\calAbar \cap \calP_\gra\right)_N^{3N}$. Thus one has 
the upper bound
$$\sum_{\substack{1\le m\le 3N\\ \rho_{\grb\grc}(m)\ge 1}}1\le \sum_{m\in (\calAbar \cap 
\calP_\gra)_N^{3N}}1\le \left|\calAbar\cap \calM\right|_N^{3N},$$
whence
$$\sum_{\grc \in \grC}\sum_{\substack{1\le m\le 3N\\ \rho_{\grb\grc}(m)\ge 1}}1\le q\left|
\calAbar\cap \calM\right|_N^{3N}.$$
The conclusion of the theorem follows on substituting these relations into 
(\ref{2.2}).
\end{proof}

The conclusion of Theorem \ref{theorem1.1} is immediate from the case $q=1$ of 
Theorem \ref{theorem2.1}, in which $\calL$, $\calM$ and $\calN$ are each taken 
to be $\dbZ$. As a first illustration of the ease with which Theorems 
\ref{theorem1.1} and \ref{theorem2.1} may be applied to concrete problems, we 
now establish a theorem which implies Theorem \ref{theorem1.2} by using the 
strategy presented first in the proof of Theorem 1.1 of \cite{slim1}. We first 
extend the notation introduced in the preamble to the statement of Theorem 
\ref{theorem1.2}. Let $\calL$ be a union of arithmetic progressions modulo $q$, 
for some natural number $q$. When $k$ is a natural number, we describe a subset 
$\calQ$ of $\dbN$ as being a {\it high-density subset of the $k$th powers 
relative to $\calL$} when (i) one has $\calQ\subseteq \{ n^k:n\in \dbN\}$, and 
(ii) for each positive number $\eps$, whenever $N$ is a natural number 
sufficiently large in terms of $\eps$, then $\langle \calQ \wedge \calL
\rangle_0^N\gg_q N^{1/k-\eps}$. Also, when $\tet>0$, we shall refer to a set 
$\calR\subseteq \dbN$ as having {\it $\calL$-complementary density growth 
exponent smaller than $\tet$} when there exists a positive number $\del$ with 
the property that, for all sufficiently large natural numbers $N$, one has 
$\left| \calRbar \cap \calL\right|_0^N<N^{\tet-\del}$.

\begin{theorem}\label{theorem2.2}
Let $\calL$, $\calM$ and $\calN$ be unions of arithmetic progressions modulo 
$q$, for some natural number $q$, and suppose that $\calN\subseteq \calL+\calM$.
 Suppose also that $\calS$ is a high-density subset of the squares relative to 
$\calL$, and that $\calA \subseteq \dbN$ has $\calM$-complementary density 
growth exponent smaller than $1$. Then, whenever $\eps>0$ and $N$ is a natural 
number sufficiently large in terms of $\eps$, one has
$$\left| {\overline{\calA+\calS}}\cap \calN\right|_{2N}^{3N}\ll_q N^{\eps-1/2}\left|
 \calAbar \cap \calM\right|_N^{3N}.$$ 
\end{theorem}

\begin{proof} Throughout the proof of this theorem, implicit constants may 
depend on $q$. Let $N$ be a large natural number, and suppose that $\calL,\calM,
\calN$ satisfy the hypotheses of the statement of the theorem. Also, let $\calS$
 be a high density subset of the squares relative to $\calL$. Then, in 
particular, there is a subset $\calT$ of $\dbN$ for which $\calS\cap \calL=
\{n^2:n\in \calT\}$. Consider also a subset $\calA$ of $\dbN$ having 
$\calM$-complementary density growth exponent smaller than $1$. Write 
$P=[N^{1/2}]$. The quantity $\Ups({\overline{\calA+\calS}}\cap \calN,\calS\cap 
\calL;N)$ counts the number of solutions of the equation
\begin{equation}\label{2.3}
n_1-n_2=x_1^2-x_2^2,
\end{equation}
with $n_1,n_2\in \left({\overline{\calA+\calS}}\cap \calN\right)_{2N}^{3N}$ and 
$x_1,x_2\in (\calT)_0^{P}$. There are plainly
$$\left|{\overline{\calA+\calS}}\cap \calN\right|_{2N}^{3N}|\calT|_0^P$$
solutions of this equation with $n_1=n_2$ and $x_1^2=x_2^2$. Given any one of the
$$O\Bigl(\left(\left|{\overline{\calA+\calS}}\cap \calN\right|_{2N}^{3N}\right)^2
\Bigr)$$
available choices of $n_1$ and $n_2$ with $n_1\ne n_2$, meanwhile, one may apply 
an elementary estimate for the divisor function to show that there are 
$O(N^\eps)$ possible choices for $x_1-x_2$ and $x_1+x_2$ satisfying (\ref{2.3}), 
whence also for $x_1$ and $x_2$. On noting that $|\calT|_0^P=
|\calS\cap \calL|_0^N$, we find that
\begin{align*}
\Ups({\overline{\calA+\calS}}\cap \calN,\calS\cap \calL;N)\ll &\,\left|
{\overline{\calA+\calS}}\cap \calN\right|_{2N}^{3N}|\calS\cap \calL|_0^N\\
&\,+N^\eps \left(\left|{\overline{\calA+\calS}}\cap \calN\right|_{2N}^{3N}
\right)^2.
\end{align*}

\par We substitute this last estimate into the conclusion of Theorem 
\ref{theorem2.1}, and thereby deduce that
\begin{align*}
\left(\langle \calS\wedge \calL\rangle_0^N\right)^2\left|
{\overline{\calA+\calS}}\cap \calN\right|_{2N}^{3N}\ll &\,\left|\calAbar\cap 
\calM\right|_N^{3N}|\calS\cap \calL|_0^N\\
&\,+N^\eps \left|\calAbar\cap \calM\right|_N^{3N}\left|{\overline{\calA+\calS}}
\cap \calN\right|_{2N}^{3N}.
\end{align*}
But since $\calS$ is a high-density subset of the squares relative to $\calL$, 
and the set $\calA$ has $\calM$-complementary density growth exponent smaller 
than $1$, then there exists a positive number $\del$ with the property that
$$N^\del \left|\calAbar\cap \calM\right|_N^{3N}\ll N^{1-\del}<(N^{1/2-\eps})^2\ll 
\left(\langle \calS\wedge \calL\rangle_0^N\right)^2.$$
In addition, one has
$$\frac{\left| \calS\cap\calL\right|_0^N}{\left( \langle \calS\wedge\calL
\rangle_0^N\right)^2}\ll N^{\eps-1/2}.$$
Thus we deduce that
$$\left|{\overline{\calA+\calS}}\cap \calN\right|_{2N}^{3N}\ll N^{\eps-1/2}\left|
\calAbar\cap \calM\right|_N^{3N}+N^{\eps-\del}\left|{\overline{\calA+\calS}}\cap 
\calN \right|_{2N}^{3N},$$
and the conclusion of the theorem follows at once.
\end{proof}

The estimate claimed in Theorem \ref{theorem1.2} follows from Theorem 
\ref{theorem2.2} on putting $q=1$ and taking $\calL$, $\calM$ and $\calN$ each 
to be $\dbZ$.

\section{Auxiliary mean values involving cubes} Before applying Theorems 
\ref{theorem1.1} or \ref{theorem2.1} to additive problems involving cubes, it is
 necessary to establish some auxiliary mean value estimates in order to bound 
the expression $\Ups({\overline{\calA+\calB}},\calB;N)$ relevant to our 
problems. This we accomplish in the present section.\par

Let $\calL$ and $\calN$ be unions of arithmetic progressions modulo $q$, for 
some natural number $q$. In addition, let $\calC$ be a high-density subset of 
the cubes relative to $\calL$, and let $\calA$ be a subset of $\dbN$. Consider a
 large natural number $N$, and write $P=N^{1/3}$. Observe first that when $s\in 
\dbN$, the quantity $\Ups({\overline{\calA+s\calC}}\cap\calN,s(\calC\cap \calL)
;N)$ is bounded above by the number of solutions of the equation
\begin{equation}\label{3.1}
n_1-n_2=\sum_{i=1}^s(x_i^3-y_i^3),
\end{equation}
with $n_1,n_2\in \left({\overline{\calA+s\calC}}\cap \calN\right)_{2N}^{3N}$ and 
$1\le x_i,y_i\le P$ $(1\le i\le s)$. Write $\calZ(N)$ for 
$\left({\overline{\calA+s\calC}}\cap \calN\right)_{2N}^{3N}$ and $Z$ for $\text{
card}(\calZ(N))$. Also, define the exponential sums
$$f(\alp)=\sum_{1\le x\le P}e(\alp x^3)\quad \text{and}\quad K(\alp)=\sum_{n\in 
\calZ(N)}e(n\alp).$$
Here, as usual, we write $e(z)$ for $e^{2\pi iz}$. Then, on considering the 
underlying diophantine equation, it follows from (\ref{3.1}) that
\begin{equation}\label{3.2}
\Ups({\overline{\calA+s\calC}}\cap\calN,s(\calC\cap \calL);N)\le \int_0^1
|f(\alp)^{2s}K(\alp)^2|\,d\alp .
\end{equation}

We begin by considering the situation in which $s=1$.

\begin{lemma}\label{lemma3.1}
One has
$$\int_0^1|f(\alp)^2K(\alp)^2|\,d\alp \ll P^\eps (P^{3/2}Z+Z^{5/3})$$
and
$$\int_0^1|f(\alp)^2K(\alp)^2|\,d\alp \ll PZ+P^{1/2+\eps}Z^{3/2}.$$
\end{lemma}

\begin{proof} We estimate the integral in question first by means of the 
Hardy-Littlewood method. When $a\in \dbZ$ and $r\in \dbN$, define the major 
arcs $\grN(r,a)$ by putting
$$\grN(r,a)=\{ \alp \in [0,1):|r\alp-a|\le P^{-2}\},$$
and then take $\grN$ to be the union of the arcs $\grN(r,a)$ with $0\le a\le 
r\le P$ and $(a,r)=1$. Also, write $\grn=[0,1)\setminus \grN$. Next, define 
$\Ups(\alp)$ for $\alp \in [0,1)$ by taking
$$\Ups(\alp)=(r+P^3|r\alp-a|)^{-1},$$
when $\alp \in \grN(r,a)\subseteq \grN$, and otherwise by putting 
$\Ups(\alp)=0$. Also, define the function $f^*(\alp)$ for $\alp\in [0,1)$ by 
taking $f^*(\alp)=P\Ups(\alp)^{1/3}$. On referring to Theorems 4.1 and 4.2, 
together with Lemma 2.8, of Vaughan \cite{Vau1997}, one readily confirms that 
the estimate
$$f(\alp)\ll f^*(\alp)+P^{1/2+\eps}$$
holds uniformly for $\alp \in \grN$. An application of Weyl's inequality 
(see Lemma 2.4 of \cite{Vau1997}), meanwhile, reveals that
$$\sup_{\alp\in \grn}|f(\alp)|\ll P^{3/4+\eps}.$$
Thus we find that, uniformly for $\alp \in [0,1)$, one has
\begin{equation}\label{3.3a}
|f(\alp)|^2\ll f^*(\alp)^2+P^{3/2+\eps},
\end{equation}
whence
\begin{equation}\label{3.4}
\int_0^1|f(\alp)^2K(\alp)^2|\,d\alp \ll P^{3/2+\eps}I_1+P^2I_2,  
\end{equation}
where
$$I_1=\int_0^1|K(\alp)|^2\,d\alp \quad \text{and}\quad I_2=\int_0^1
\Ups(\alp)^{2/3}|K(\alp)|^2\,d\alp .$$

\par By Parseval's identity, one has $I_1=Z$. Meanwhile, an application of 
H\"older's inequality combined with Lemma 2 of \cite{Bru1988} shows that
\begin{align*}
I_2&\ll \Bigl( \int_0^1 \Ups(\alp)|K(\alp)|^2\,d\alp \Bigr)^{2/3}\Bigl( \int_0^1
|K(\alp)|^2\,d\alp\Bigr)^{1/3}\\
&\ll \left( P^{\eps-3}(PZ+Z^2)\right)^{2/3}Z^{1/3}.
\end{align*}
On substituting these estimates into (\ref{3.4}), we deduce that
$$\int_0^1|f(\alp)^2K(\alp)^2|\,d\alp \ll P^{3/2+\eps}Z+P^\eps (P^{2/3}Z+Z^{5/3}),$$
and the first conclusion of the lemma follows.\par

In order to confirm the second estimate, we begin by considering the underlying 
diophantine equation. One finds that the mean value in question is bounded above
 by $\sum_mQ(m)^2$, where $Q(m)$ denotes the number of solutions of the equation
 $x^3+n=m$, with $1\le x\le P$ and $n\in \calZ(N)$. The desired conclusion 
therefore follows from a trivial modification of Theorem 1 of \cite{Dav1942} 
(see, for example, Theorem 6.2 of \cite{Vau1997} with $k=3$, $j=1$ and $\nu=1$, 
or the case $k=3$ and $j=1$ of Lemma \ref{lemma5.1} below).
\end{proof}

Next we consider the situation with $s=2$.

\begin{lemma}\label{lemma3.2} One has
$$\int_0^1|f(\alp)^4K(\alp)^2|\,d\alp \ll P^2Z+P^{11/6+\eps}Z^2,$$
and
$$\int_0^1|f(\alp)^4K(\alp)^2|\,d\alp \ll P^\eps(P^{5/2}Z+P^2Z^{3/2}+PZ^2).$$
\end{lemma}

\begin{proof} On making use of a bound of Parsell based on the methods of Hooley
 (see Lemma 2.1 of \cite{Par2000}, and also \cite{Hoo1978}), the argument of the
 proof of Lemma 10.3 of \cite{slim2} supplies the first bound claimed in the 
lemma. We refer the reader to the discussion on pages 420 and 447 of 
\cite{slim2} for amplification on this matter.\par

For the second bound, we again apply the Hardy-Littlewood method, and for this 
purpose it is convenient to employ the notation introduced in the course of the 
proof of Lemma \ref{lemma3.1}. First, from (\ref{3.3a}), we deduce that
\begin{equation}\label{3.w1}
\int_0^1|f(\alp)^4K(\alp)^2|\,d\alp \ll P^{3/2+\eps}I_3+P^2I_4,  
\end{equation}
where
$$I_3=\int_0^1|f(\alp)^2K(\alp)^2|\,d\alp \quad \text{and}\quad I_4=\int_0^1
\Ups(\alp)^{2/3}|f(\alp)^2K(\alp)^2|\,d\alp .$$
From the second estimate of Lemma \ref{lemma3.1}, one has
$I_3\ll PZ+P^{1/2+\eps}Z^{3/2}$. Meanwhile, an application of Schwarz's inequality 
combined with Lemma 2 of \cite{Bru1988} on this occasion shows that
\begin{align*}
I_4&\ll \Bigl( \int_0^1 \Ups(\alp)^{4/3}|K(\alp)|^2\,d\alp \Bigr)^{1/2}
\Bigl( \int_0^1|f(\alp)^4K(\alp)^2|\,d\alp\Bigr)^{1/2}\\
&\ll \left( P^{\eps-3}(PZ+Z^2)\right)^{1/2}\Bigl( \int_0^1|f(\alp)^4K(\alp)^2|\,
d\alp\Bigr)^{1/2}.
\end{align*}
On substituting these estimates into (\ref{3.w1}), we conclude that
$$\int_0^1|f(\alp)^4K(\alp)^2|\,d\alp \ll P^{3/2+\eps}(PZ+P^{1/2}Z^{3/2})+P^4\left(
P^{\eps-3}(PZ+Z^2)\right),$$
and the second estimate of the lemma follows.
\end{proof}

Although we are able to avoid explicit reference to the case $s=3$ within this 
paper, it is useful for future reference to provide additional bounds of utility
 in this situation.

\begin{lemma}\label{lemma3.3} One has
$$\int_0^1|f(\alp)^6K(\alp)^2|\,d\alp \ll P^\eps(P^4Z+P^3Z^2),$$
and
$$\int_0^1|f(\alp)^6K(\alp)^2|\,d\alp \ll P^\eps(P^{7/2}Z+P^{10/3}Z^2).$$
\end{lemma}

\begin{proof} We begin by observing that the first estimate of the lemma is 
essentially the bound supplied by Lemma 6.2 of \cite{slim2}, and indeed that the
 argument employed to establish the latter suffices for our purposes. For the 
second bound, we once again apply the Hardy-Littlewood method, and employ the 
notation introduced in the course of the proof of Lemma \ref{lemma3.1}. First, 
from (\ref{3.3a}), we deduce that
\begin{equation}\label{3.w2}
\int_0^1|f(\alp)^6K(\alp)^2|\,d\alp \ll P^{3/2+\eps}I_5+P^2I_6,  
\end{equation}
where
$$I_5=\int_0^1|f(\alp)^4K(\alp)^2|\,d\alp \quad \text{and}\quad I_6=\int_0^1
\Ups(\alp)^{2/3}|f(\alp)^4K(\alp)^2|\,d\alp .$$
From the first estimate of Lemma \ref{lemma3.2}, one has $I_5\ll P^2Z+P^{11/6+\eps}
Z^2$. Meanwhile, an application of H\"older's inequality combined with a 
routine computation shows that
\begin{align*}
I_6&\ll \Bigl( K(0)^2\int_0^1 \Ups(\alp)^2\,d\alp \Bigr)^{1/3}\Bigl( \int_0^1
|f(\alp)^6K(\alp)^2|\,d\alp\Bigr)^{2/3}\\
&\ll \left( P^{\eps-3}Z^2\right)^{1/3}\Bigl( \int_0^1|f(\alp)^6K(\alp)^2|\,d\alp
\Bigr)^{2/3}.
\end{align*}
On substituting these estimates into (\ref{3.w2}), we conclude that
$$\int_0^1|f(\alp)^6K(\alp)^2|\,d\alp \ll P^{3/2+\eps}(P^2Z+P^{11/6}Z^2)+P^6\left(
P^{\eps-3}Z^2\right),$$
and the second bound of the lemma follows.
\end{proof}

\section{Additive problems involving cubes} Our goal in this section is the 
proof of Theorem \ref{theorem1.3}, and this we achieve in Theorem 
\ref{theorem4.1} below. It is useful in applications to have available 
conclusions analogous to those of Theorem \ref{theorem1.3}, though with 
additional congruence conditions present. We therefore spend a little extra 
effort to establish more general conclusions of this type.

\begin{theorem}\label{theorem4.1}
Let $\calL$, $\calM$ and $\calN$ be unions of arithmetic progressions modulo 
$q$, for some natural number $q$. Suppose also that $\calC$ is a high-density 
subset of the cubes relative to $\calL$, and that $\calA\subseteq \dbN$ has 
$\calM$-complementary density growth exponent smaller than $\tet$, for some 
positive number $\tet$. Then, whenever $\eps>0$ and $N$ is a natural number 
sufficiently large in terms of $\eps$, one has the following estimates:
\vskip.05cm
\noindent (a) when $\calN\subseteq \calL+\calM$, then without any condition on 
$\tet$, one has
$$\left|{\overline{\calA+\calC}}\cap \calN\right|_{2N}^{3N}\ll_q N^{\eps -1/6}
\left|\calAbar\cap \calM\right|_N^{3N}+N^{\eps -2}\left( \left|\calAbar\cap \calM
\right|_N^{3N}\right)^3,$$
and
$$\left|{\overline{\calA+\calC}}\cap \calN\right|_{2N}^{3N}\ll_q N^{\eps -1/3}\left|
\calAbar \cap \calM\right|_N^{3N}+N^{\eps -1}\left( \left|\calAbar\cap \calM
\right|_N^{3N}\right)^2;$$
\vskip.05cm
\noindent (b) when $\calN\subseteq 2\calL+\calM$, then provided that $\tet\le 
1$, one has
$$\left|{\overline{\calA+2\calC}}\cap \calN\right|_{2N}^{3N}\ll_q N^{\eps -1/2}
\left|\calAbar\cap \calM\right|_N^{3N}+N^{\eps-4/3}\left( \left| \calAbar \cap 
\calM\right|_N^{3N}\right)^2,$$
and when $\tet \le \frac{13}{18}$, one has
$$\left|{\overline{\calA+2\calC}}\cap \calN\right|_{2N}^{3N}\ll_q N^{\eps -2/3}
\left|\calAbar\cap \calM\right|_N^{3N};$$
\vskip.05cm
\noindent (c) when $\calN\subseteq 3\calL+\calM$, then provided that 
$\tet\le 1$, one has
$$\left|{\overline{\calA+3\calC}}\cap \calN\right|_{4N}^{6N}\ll_q N^{\eps -5/3}\left(
\left|\calAbar \cap \calM\right|_N^{6N}\right)^2,$$
and when $\tet \le \frac{8}{9}$, one has
$$\left|{\overline{\calA+3\calC}}\cap \calN\right|_{4N}^{6N}\ll_q N^{\eps -5/6}\left|
\calAbar \cap \calM\right|_N^{6N}.$$
\end{theorem}

\begin{proof} Let $N$ be a large natural number, write $P=N^{1/3}$, and suppose 
that $\calA, \calC, \calL, \calM, \calN$ satisfy the hypotheses of the 
statement of the theorem. In what follows, implicit constants may depend on $q$.
 We begin by considering the estimates claimed in part (a) of the theorem. 
Observe that since $\calC$ is a high-density subset of the cubes relative to 
$\calL$, then $\langle \calC \wedge \calL\rangle_0^N\gg N^{1/3-\eps}$, and hence it
 follows from Theorem \ref{theorem2.1} that
$$\left(N^{1/3-\eps}\left|{\overline{\calA+\calC}}\cap \calN\right|_{2N}^{3N}
\right)^2\ll \left| \calAbar \cap \calM\right|_N^{3N}\Ups({\overline{\calA+
\calC}}\cap \calN,\calC\cap \calL;N).$$
Consequently, on making use of the relation (\ref{3.2}) with $s=1$ and the first
 estimate of Lemma \ref{lemma3.1}, we deduce that
\begin{align*}
N^{2/3-\eps}\left| {\overline{\calA+\calC}}\cap \calN\right|_{2N}^{3N}\ll &\,N^{1/2+
\eps}\left|\calAbar\cap \calM \right|_N^{3N}\\
&\,+N^\eps \left|\calAbar\cap \calM \right|_N^{3N}\left(\left|{\overline{\calA+
\calC}}\cap \calN\right|_{2N}^{3N}\right)^{2/3}.
\end{align*}
From here it follows that
$$\left|{\overline{\calA+\calC}}\cap \calN\right|_{2N}^{3N}\ll N^{\eps -1/6}\left|
\calAbar \cap \calM\right|_N^{3N}+\left(N^{\eps-2/3}\left|\calAbar\cap \calM
\right|_N^{3N}\right)^3.$$
The first estimate of part (a) is thus confirmed.\par

If instead we apply (\ref{3.2}) with $s=1$ and the second estimate of Lemma 
\ref{lemma3.1}, then we obtain the bound
\begin{align*}
N^{2/3-\eps}\left| {\overline{\calA+\calC}}\cap \calN\right|_{2N}^{3N}\ll &\,N^{1/3}
\left|\calAbar\cap \calM \right|_N^{3N}\\
&\,+N^{1/6+\eps}\left|\calAbar\cap \calM \right|_N^{3N}\left(\left|{\overline{
\calA+\calC}}\cap \calN\right|_{2N}^{3N}\right)^{1/2},
\end{align*}
whence
$$\left|{\overline{\calA+\calC}}\cap \calN\right|_{2N}^{3N}\ll N^{\eps -1/3}\left|
\calAbar \cap \calM\right|_N^{3N}+\left(N^{\eps-1/2}\left|\calAbar\cap \calM
\right|_N^{3N}\right)^2.$$
This delivers the second estimate of part (a).\par

We next consider the estimates claimed in part (b). Observe first that by 
applying an elementary divisor function estimate, one confirms that the number 
of representations of a positive integer $n$, as the sum of two positive 
integral cubes, is $O(n^\eps)$. It follows that when $\calC$ is a high-density 
subset of the cubes relative to $\calL$, then $\langle 2\calC \wedge 2\calL
\rangle_0^N\gg N^{2/3-\eps}$. We therefore find from Theorem \ref{theorem2.1} that
$$\left(N^{2/3-\eps}\left|{\overline{\calA+2\calC}}\cap \calN\right|_{2N}^{3N}
\right)^2\ll \left| \calAbar \cap \calM\right|_N^{3N}\Ups({\overline{\calA+
2\calC}}\cap \calN,2(\calC\cap \calL);N).$$
Consequently, on making use of the relation (\ref{3.2}) with $s=2$ and the 
second estimate of Lemma \ref{lemma3.2}, we deduce that
\begin{align*}
N^{4/3-\eps}\left| {\overline{\calA+2\calC}}\cap \calN\right|_{2N}^{3N}\ll 
&\,N^{5/6+\eps}\left|\calAbar\cap \calM \right|_N^{3N}\\
&\,+N^{2/3+\eps}\left|\calAbar\cap \calM \right|_N^{3N}\left(\left|{\overline{
\calA+2\calC}}\cap \calN\right|_{2N}^{3N}\right)^{1/2}\\
&\,+N^{1/3+\eps}\left|\calAbar\cap \calM \right|_N^{3N}\left|{\overline{\calA+
2\calC}}\cap \calN\right|_{2N}^{3N}.
\end{align*}
We therefore deduce that when $\calA$ has $\calM$-complementary density growth 
exponent smaller than $1$, then there is a positive number $\del$ with the 
property that
\begin{align*}
\left|{\overline{\calA+2\calC}}\cap \calN\right|_{2N}^{3N}\ll &\,N^{\eps -1/2}\left|
\calAbar \cap \calM\right|_N^{3N}+N^{\eps-\del}\left|{\overline{\calA+2\calC}}\cap 
\calN\right|_{2N}^{3N}\\
&\,+\left(N^{\eps-2/3}\left|\calAbar\cap \calM \right|_N^{3N}\right)^2.
\end{align*}
The first estimate of part (b) now follows.\par

If instead we apply (\ref{3.2}) with $s=2$ and the first estimate of Lemma 
\ref{lemma3.2}, then we obtain the bound
\begin{align*}
N^{4/3-\eps}\left| {\overline{\calA+2\calC}}\cap \calN\right|_{2N}^{3N}\ll &\,N^{2/3}
\left|\calAbar\cap \calM \right|_N^{3N}\\
&\,+N^{11/18+\eps}\left|\calAbar\cap \calM \right|_N^{3N}\left|{\overline{\calA+
\calC}}\cap \calN\right|_{2N}^{3N}.
\end{align*}
Thus, when $\calA$ has $\calM$-complementary density growth exponent smaller 
than $\frac{13}{18}$, then there is a positive number $\del$ with the property 
that
$$\left|{\overline{\calA+2\calC}}\cap \calN\right|_{2N}^{3N}\ll N^{\eps -2/3}\left|
\calAbar \cap \calM\right|_N^{3N}+N^{\eps-\del}\left|{\overline{\calA+2\calC}}\cap 
\calN\right|_{2N}^{3N}.$$
The second estimate of part (b) is now immediate.\par

Finally, we turn our attention to the estimates claimed in part (c) of the 
theorem. We take $\calN_0=2\calL+\calM$, so that $\calN_0$ is a union of 
arithmetic progressions modulo $q$. Then the first estimate of part (b) implies 
that when $\calA$ has $\calM$-complementary density growth exponent smaller than
 $1$, then
$$\left|{\overline{\calA+2\calC}}\cap \calN_0\right|_{4N}^{6N}\ll N^{\eps -1/2}
\left|\calAbar \cap \calM\right|_{2N}^{6N}+N^{\eps -4/3}\left( \left|\calAbar \cap 
\calM\right|_{2N}^{6N}\right)^2.$$
In particular, there is a positive number $\del$ for which
$$\left|{\overline{\calA+2\calC}}\cap \calN_0\right|_{4N}^{6N}\ll N^{\eps-1/2}
(N^{1-\del})+N^{\eps-4/3}(N^{1-\del})^2\ll N^{2/3-\del}.$$
Thus, by summing over dyadic intervals, it follows that $\calA+2\calC$ has 
$\calN_0$-complementary density growth exponent smaller than $\frac{2}{3}$. On 
noting that $\calN\subseteq \calN_0+\calL=3\calL+\calM$, we deduce from the 
second estimate of part (a) that
\begin{align*}
\left|{\overline{\calA+3\calC}}\cap \calN\right|_{4N}^{6N}&\le \left|{\overline{(
\calA+2\calC)+\calC}}\cap (\calN_0+\calL)\right|_{4N}^{6N}\\
&\ll N^{\eps-1/3}\left|{\overline{\calA+2\calC}}\cap \calN_0\right|_{2N}^{6N}
+N^{\eps-1}\left( \left| {\overline{\calA+2\calC}}\cap \calN_0\right|_{2N}^{6N}
\right)^2\\
&\ll N^{\eps-1/3}\left| {\overline{\calA+2\calC}}\cap \calN_0\right|_{2N}^{6N},
\end{align*}
and hence
$$\left|{\overline{\calA+3\calC}}\cap \calN\right|_{4N}^{6N}\ll N^{\eps-5/6}\left|
 \calAbar \cap \calM\right|_N^{6N}+\left( N^{\eps-5/6}\left| \calAbar \cap \calM
\right|_N^{6N}\right)^2.$$
The first estimate of part (c) is now immediate.\par

Next we take $\calN_1=\calL+\calM$, so that $\calN_1$ is a union of arithmetic 
progressions modulo $q$. The first estimate of part (a) implies that when 
$\calA$ has $\calM$-complementary density growth exponent smaller than 
$\frac{8}{9}$, then
$$\left|{\overline{\calA+\calC}}\cap \calN_1\right|_{4N}^{6N}\ll N^{\eps -1/6}
\left|\calAbar \cap \calM\right|_{2N}^{6N}+N^{\eps -2}\left( \left|\calAbar \cap 
\calM\right|_{2N}^{6N}\right)^3.$$
In particular, there is a positive number $\del$ for which
$$\left|{\overline{\calA+\calC}}\cap \calN_1\right|_{4N}^{6N}\ll N^{\eps-1/6}
(N^{8/9-\del})+N^{\eps-2}(N^{8/9-\del})^3\ll N^{13/18-\del+\eps}.$$
Thus, by summing over dyadic intervals, it follows that the set $\calA+\calC$ 
has $\calN_1$-complementary density growth exponent smaller than 
$\frac{13}{18}$. On noting that $\calN\subseteq \calN_1+2\calL=3\calL+\calM$, we
 deduce from the second estimate of part (b) that
\begin{align*}
\left|{\overline{\calA+3\calC}}\cap \calN\right|_{4N}^{6N}&\le \left|{\overline{(
\calA+\calC)+2\calC}}\cap (\calN_1+2\calL)\right|_{4N}^{6N}\\
&\ll N^{\eps-2/3}\left|{\overline{\calA+\calC}}\cap \calN_1\right|_{2N}^{6N}\\
&\ll N^{\eps-5/6}\left| \calAbar\cap \calM\right|_N^{6N}+N^{\eps-8/3}\left(\left| 
\calAbar\cap \calM\right|_N^{6N}\right)^3,
\end{align*}
and hence, for some positive number $\del$, one has
$$\left|{\overline{\calA+3\calC}}\cap \calN\right|_{4N}^{6N}\ll N^{\eps-5/6}\left|
 \calAbar \cap \calM\right|_N^{6N}+N^{\eps-\del}.$$
The second estimate of part (c) now follows at once.\par
\end{proof}

On taking $q=1$ and $\calL, \calM,\calN$ each to be $\dbZ$, the various 
conclusions of Theorem \ref{theorem4.1} suffice to establish Theorem 
\ref{theorem1.3}.

\section{Consequences for sums of cubes}
The estimates contained in Theorems \ref{theorem1.4} and \ref{theorem1.5} are 
straightforward corollaries of Theorems \ref{theorem1.3} and \ref{theorem4.1}, 
as we now demonstrate. 

\begin{proof}[The proof of Theorem \ref{theorem1.4}]
The estimate for $E_4(N)$ recorded in the statement of Theorem \ref{theorem1.4} 
is, as mentioned earlier, simply a restatement of Theorem 1.3 of \cite{Woo2000}.
 Write $\nu=2982+56\sqrt{2833}$. We set $\calC=\{ n^3:n\in \dbN\}$ and 
$\calA=4\calC$, and note that this first estimate yields the bound 
$\left|\calAbar\right|_N^{3N}\le E_4(3N)\ll N^{37/42-\tau}$, for any positive number
 $\tau$ with $\tau^{-1}>\nu$. An application of the first estimate of Theorem 
\ref{theorem1.3}(a) yields the bound
\begin{align*}
\left|{\overline{\calA+\calC}}\right|_{2N}^{3N}&\ll N^{\eps-1/6}E_4(3N)+N^{\eps-2}
(E_4(3N))^3\\
&\ll N^{5/7-\tau+\eps}+N^{9/14-3\tau+\eps}.
\end{align*}
Write $\lceil \tet \rceil$ for the least integer not smaller than $\tet$, and 
define the integers $N_j$ for $j\ge 0$ by means of the iterative formula
\begin{equation}\label{4.w1}
N_0=\lceil {\textstyle{\frac{1}{2}}}N\rceil ,\quad N_{j+1}=\lceil 
{\textstyle{\frac{2}{3}}}N_j\rceil\qquad (j\ge 0).
\end{equation}
In addition, define $J$ to be the least positive integer with the property that 
$N_J=2$, and note that $J=O(\log N)$. Then, whenever $\tau_1$ is a positive 
number with $\tau_1^{-1}>\nu$, one has
$$E_5(N)\le 3+\sum_{j=1}^J\left|{\overline{\calA+\calC}}\right|_{2N_j}^{3N_j}\ll 
N^{5/7-\tau_1}.$$

\par Next we put $\calA=5\calC$, and note that the estimate just provided 
implies that $\left|\calAbar\right|_N^{3N}\le E_5(3N)\ll N^{5/7-\tau}$, for any 
positive number $\tau$ with $\tau^{-1}>\nu$. We now apply the second estimate of 
Theorem \ref{theorem1.3}(a), obtaining the bound
\begin{align*}
\left|{\overline{\calA+\calC}}\right|_{2N}^{3N}&\ll N^{\eps-1/3}E_5(3N)+N^{\eps-1}
(E_5(3N))^2\\
&\ll N^{8/21-\tau+\eps}+N^{3/7-2\tau+\eps}.
\end{align*}
Thus, whenever $\tau_1$ is a positive number with $\tau_1^{-1}>\nu$, one deduces 
that
$$E_6(N)\le 3+\sum_{j=1}^J\left|{\overline{\calA+\calC}}\right|_{2N_j}^{3N_j}\ll 
N^{3/7-2\tau_1}.$$
This completes the proof of Theorem \ref{theorem1.4}.
\end{proof}

\begin{proof}[The proof of Theorem \ref{theorem1.5}]
When $p$ is a prime number exceeding $7$, one has
$$p^3\equiv 1\mmod{2},\quad p^3\equiv \pm 1\mmod{9}\quad \text{and}\quad 
p^3\equiv \pm 1\mmod{7}.$$
If we put $\calC=\{ p^3:\text{$p$ prime and $p>7$}\}$, then it follows that for 
$s\ge 5$, one has $s\calC\subseteq \calN_s$, where we write
\begin{align*}
\calN_5&=\{ n\in \dbN : \text{$n\equiv 1\mmod{2}$, $n\not\equiv 0,\pm 
2\mmod{9}$, $n\not\equiv 0\mmod{7}$}\},\\
\calN_6&=\{ n\in \dbN : \text{$n\equiv 0\mmod{2}$, $n\not\equiv \pm 1\mmod{9}$}
\},\\
\calN_7&=\{ n\in \dbN : \text{$n\equiv 1\mmod{2}$, $n\not\equiv 0\mmod{9}$}\},\\
\calN_s&=\{ n\in \dbN : \text{$n\equiv s\mmod{2}$}\}\quad (s\ge 8).
\end{align*}
Observe that our definition of the exceptional set in the Waring-Goldbach 
problem for cubes, given in the preamble to Theorem \ref{theorem1.5}, may be 
recovered by putting $\calE_s(N)=\left| {\overline{s\calC}}\cap \calN_s
\right|_0^N$, a definition that we now extend to all natural numbers $s$.\par

Write
$$\calL=\{ l\in \dbN : \text{$l\equiv 1\mmod{2}$, $l\equiv \pm 1\mmod{9}$, $l
\equiv \pm 1\mmod{7}$}\}.$$
Then $\calL$ and $\calN_s$ $(s\ge 5)$ are unions of arithmetic progressions 
modulo $126$ satisfying the condition that $\calN_{s+1}=\calL+\calN_s$ and 
$\calN_{s+2}=2\calL+\calN_s$ $(s\ge 5)$. Moreover, it follows from the Prime 
Number Theorem in arithmetic progressions that $\langle \calC \wedge \calL
\rangle_0^N\gg N^{1/3}(\log N)^{-1}$, so that $\calC$ is a high-density subset of 
the cubes relative to $\calL$. Observe next that the proof underlying the first 
estimate of Theorem 1 of Kumchev \cite{Kum2005} shows that $\calE_5(N)\ll 
N^{79/84-\nu}$, for some positive number $\nu$. But $\calE_s(N)=\left| {\overline{
s\calC}}\cap \calN_s\right|_0^N$, and so it follows from the first bound of 
Theorem \ref{theorem4.1}(a) that
\begin{align*}
\left|{\overline{5\calC+\calC}}\cap \calN_6\right|_{2N}^{3N}&\ll N^{\eps-1/6}\left|
{\overline{5\calC}}\cap \calN_5\right|_N^{3N}+N^{\eps-2}\left(\left|{\overline{
5\calC}}\cap\calN_5\right|_N^{3N}\right)^3\\
&\ll N^{\eps-1/6}\calE_5(3N)+N^{\eps-2}\left( \calE_5(3N)\right)^3\\
&\ll N^{65/84-\nu+\eps}+N^{23/28-3\nu+\eps}.
\end{align*}
Consequently, on making use again of the notation introduced in (\ref{4.w1}), it
 follows that
$$\calE_6(N)\le 3+\sum_{j=1}^J\left|{\overline{5\calC+\calC}}\cap \calN_6
\right|_{2N_j}^{3N_j}\ll N^{23/28}.$$

\par Likewise, from the first bound of Theorem \ref{theorem4.1}(b), one finds 
that
\begin{align*}
\left|{\overline{5\calC+2\calC}}\cap \calN_7\right|_{2N}^{3N}&\ll N^{\eps-1/2}\left|
{\overline{5\calC}}\cap \calN_5\right|_N^{3N}+N^{\eps-4/3}\left(\left|{\overline{
5\calC}}\cap \calN_5\right|_N^{3N}\right)^2\\
&\ll N^{\eps-1/2}\calE_5(3N)+N^{\eps-4/3}\left( \calE_5(3N)\right)^2\\
&\ll N^{37/84-\nu+\eps}+N^{23/42-2\nu+\eps}.
\end{align*}
Consequently, one has
$$\calE_7(N)\le 3+\sum_{j=1}^J\left|{\overline{5\calC+2\calC}}\cap \calN_7
\right|_{2N_j}^{3N_j}\ll N^{23/42}.$$

\par Finally, from the first bound of Theorem \ref{theorem4.1}(c), one finds 
that
\begin{align*}
\left|{\overline{5\calC+3\calC}}\cap \calN_8\right|_{4N}^{6N}&\ll N^{\eps-5/3}
\left( \left|{\overline{5\calC}}\cap \calN_5\right|_N^{6N}\right)^2\\
&\ll N^{\eps-5/3}(\calE_5(6N))^2\ll N^{3/14-2\nu+\eps}.
\end{align*}
Then, one has
$$\calE_8(N)\le 3+\sum_{j=1}^J\left|{\overline{5\calC+3\calC}}\cap \calN_8
\right|_{2N_j}^{3N_j}\ll N^{3/14}.$$
This completes the proof of Theorem \ref{theorem1.5}.
\end{proof}

\section{Sums of biquadrates}
Our bounds for $Y_{10}(N)$ and $Y_{11}(N)$ depend on a generalisation of the 
second estimate of Lemma \ref{lemma3.1} to $k$th powers, with $k\ge 3$. This 
variant of Davenport's bound we record in Lemma \ref{lemma5.1} below. We first 
introduce some further notation. Let $k$ be a natural number with $k\ge 3$, let 
$\calK$ be a high-density subset of the $k$th powers, and write
$$g(\alp)=\sum_{x\in (\calK)_0^N}e(\alp x).$$
Also, let $\calZ$ be a subset of $\dbN$, write $Z=|\calZ|_0^N$, and define the 
exponential sum
$$K(\alp)=\sum_{n\in (\calZ)_0^N}e(n\alp).$$

\begin{lemma}\label{lemma5.1}
Let $k$ be a natural number with $k\ge 3$, and suppose that $1\le j\le k-2$. Let
 $N$ be a large natural number, and put $P=N^{1/k}$. Then one has
\begin{equation}\label{5.1}
\int_0^1|g(\alp)^{2^j}K(\alp)^2|\,d\alp \ll P^{2^j-1}Z+P^{2^j-\frac{1}{2}j-1+\eps}Z^{3/2}.
\end{equation}  
\end{lemma}

\begin{proof} We begin with the trivial observation that, on considering the 
underlying diophantine equations, one has
$$\int_0^1|g(\alp)^{2^j}K(\alp)^2|\,d\alp \le \int_0^1|f(\alp)^{2^j}K(\alp)^2|\,
d\alp ,$$
where
$$f(\alp)=\sum_{1\le x\le P}e(\alp x^k).$$
Next, let $\Del_j$ denote the $j$th iterate of the forward differencing 
operator, so that whenever $\phi$ is a function of a real variable $z$, one has
$$\Del_1(\phi(z);h)=\phi(z+h)-\phi(z),$$
and when $J\ge 1$, then
$$\Del_{J+1}(\phi(z);h_1,\dots ,h_{J+1})=\Del_1(\Del_J(\phi(z);h_1,\dots 
,h_J);h_{J+1}).$$
It follows via a modest computation that when $1\le J\le k$, then
$$\Del_J(z^k;\bfh)=h_1\dots h_Jp_J(z;\bfh),$$
where $p_J$ is a homogeneous polynomial in $z$ and $\bfh$ of total degree $k-J$,
 in which the coefficient of $z^{k-J}$ is $k!/(k-J)!$. By the Weyl differencing 
lemma (see, for example, Lemma 2.3 of \cite{Vau1997}), one has
$$|f(\alp)|^{2^j}\le (2P)^{2^j-j-1}\sum_{|h_1|<P}\dots \sum_{|h_j|<P}T_j,$$
where
$$T_j=\sum_{x\in I_j}e(\alp h_1\dots h_jp_j(x;\bfh)),$$
and $I_j=I_j(\bfh)$ denotes an interval of integers, possibly empty, contained 
in $[1,P]$. On recalling the definition of $K(\alp)$, therefore, it follows 
from orthogonality that the integral on the left hand side of (\ref{5.1}) is 
bounded above by the number of integral solutions of the equation
\begin{equation}\label{5.w}
h_1\dots h_jp_j(z;\bfh)=n_1-n_2,
\end{equation}
with $|h_i|<P$ $(1\le i\le j)$, $1\le z\le P$ and $n_l\in (\calZ)_0^N$ 
$(l=1,2)$, and with each solution being counted with weight $(2P)^{2^j-j-1}$.\par

There are $O(P^j)$ possible choices for $z$ and $\bfh$ with $h_1\dots h_j=0$. 
Given any one such choice, the equation (\ref{5.w}) implies that $n_1=n_2$. 
Next, when $m$ is a natural number and $|h_i|<P$ $(1\le i\le j)$, write 
$\rho(m;\bfh)$ for the number of integral solutions of the equation
\begin{equation}\label{5.1b}
h_1\dots h_jp_j(z;\bfh)+n=m,
\end{equation}
with $1\le z\le P$ and $n\in (\calZ)_0^N$. Then on isolating the solutions with 
$h_1\dots h_j=0$ discussed earlier, we have shown at this point that
\begin{equation}\label{5.1a}
\int_0^1|g(\alp)^{2^j}K(\alp)^2|\,d\alp \ll P^{2^j-1}Z+P^{2^j-j-1}S_1,  
\end{equation}
where
$$S_1=\sum_{m,\bfh}\rho(m;\bfh),$$
in which the summation is over $m$ and $\bfh$ with $m\in (\calZ)_0^N$ and 
$1\le |h_i|<P$ $(1\le i\le j)$. Furthermore, an application of Cauchy's 
inequality reveals that
$$\Bigl( \sum_{m,\bfh}\rho(m;\bfh)\Bigr)^2\le \Bigl( \sum_{m,\bfh}1\Bigr) S_2,$$
where
$$S_2=\sum_{m,\bfh}\rho(m;\bfh)^2.$$
We therefore see that
\begin{equation}\label{5.2}
S_1\ll (P^jZ)^{1/2}S_2^{1/2}.
\end{equation}

\par Next observe that $S_2$ counts the number of integral solutions of the 
system of equations
\begin{equation}\label{5.3}
h_1\dots h_jp_j(z_1;\bfh)+n_1=h_1\dots h_jp_j(z_2;\bfh)+n_2=n_3,
\end{equation}
with $1\le z_1,z_2\le P$, $1\le |h_i|<P$ $(1\le i\le j)$ and $n_l\in (\calZ)_0^N$
 $(l=1,2,3)$. When $1\le j\le k-2$, the expression
\begin{equation}\label{5.y}
h_1\dots h_j(p_j(z_1;\bfh)-p_j(z_2;\bfh))
\end{equation}
is a non-constant polynomial in $z_1$ and $z_2$. In particular, given a solution
 $\bfz,\bfh,\bfn$ counted by $S_2$ with $n_1=n_2$, then for each fixed choice of
 $z_1$ and $\bfh$, there are $O(1)$ possible choices for $z_2$. The number of 
solutions of this type is therefore bounded above by a fixed positive multiple 
of the number of integral solutions of the equation (\ref{5.1b}) with 
$1\le z\le P$, $1\le |h_i|<P$ $(1\le i\le j)$ and $n,m\in (\calZ)_0^N$. We 
conclude, therefore, that the number of solutions $\bfz,\bfh,\bfn$ counted by 
$S_2$ with $n_1=n_2$ is $O(S_1)$.\par

Now consider a solution $\bfz,\bfh,\bfn$ counted by $S_2$ with $n_1\ne n_2$. 
Since the polynomial (\ref{5.y}) is divisible by $z_1-z_2$, one finds that 
$h_1,\dots ,h_j$ and $z_1-z_2$ are all divisors of the non-zero integer 
$n_1-n_2$. Given any one of the $O(Z^2)$ possible choices for $n_1$ and $n_2$ 
with $n_1\ne n_2$, therefore, an elementary estimate for the divisor function 
confirms that the number of choices for $\bfh$ and $z_1-z_2$ counted by $S_2$ 
is at most $O(P^\eps)$. Fixing any one such choice of $\bfh$ and $d=z_1-z_2$, 
and noting that $1\le j\le k-2$, one finds from (\ref{5.3}) that $z_1$ is 
determined from the polynomial equation
$$h_1\dots h_j(p_j(z_1;\bfh)-p_j(z_1-d;\bfh))=n_2-n_1,$$
to which there are $O(1)$ solutions. Given any one such solution, the value of 
$z_2=z_1-d$ is fixed, as is the value of $n_3$ from (\ref{5.3}). Thus we 
conclude that there are $O(P^\eps Z^2)$ solutions of this type.\par

At this point, we have shown that $S_2\ll S_1+P^\eps Z^2$, whence from 
(\ref{5.2}), we have
$$S_1\ll (P^jZ)^{1/2}(S_1+P^\eps Z^2)^{1/2}.$$
Consequently, we derive the upper bound
$$S_1\ll P^jZ+P^{\frac{1}{2}j+\eps}Z^{3/2},$$
and the conclusion of the lemma follows from (\ref{5.1a}).
\end{proof}

We note that the estimate supplied by Lemma \ref{lemma5.1} is related to that 
found in an intermediate step of the proof of Theorem 1 of \cite{Dav1942}.\par

We are now equipped to discuss additive problems involving biquadrates. Rather 
than constraining ourselves to the proof of Theorem \ref{theorem1.6}, we again 
record a more general estimate.

\begin{theorem}\label{theorem6.2w}
Let $\calL$, $\calM$ and $\calN$ be unions of arithmetic progressions modulo 
$q$, for some natural number $q$, and suppose that $\calN\subseteq 
2\calL+\calM$. Suppose also that $\calB$ is a high-density subset of the 
biquadrates relative to $\calL$, and that $\calA\subseteq \dbN$. Then, whenever 
$\eps>0$ and $N$ is a natural number sufficiently large in terms of $\eps$, one 
has
$$\left|{\overline{\calA+2\calB}}\cap \calN\right|_{2N}^{3N}\ll_q N^{\eps -1/4}
\left|\calAbar\cap \calM\right|_N^{3N}+N^{\eps -1}\left( \left|\calAbar\cap \calM
\right|_N^{3N}\right)^2.$$
\end{theorem}

\begin{proof} Let $N$ be a large natural number, and suppose that 
$\calL,\calM,\calN$ satisfy the hypotheses of the statement of the theorem. 
Also, let $\calB$ be a high density subset of the biquadrates relative to 
$\calL$. Then, in particular, there is a subset $\calT$ of $\dbN$ for which 
$\calB\cap \calL=\{n^4:n\in \calT\}$. Consider also a subset $\calA$ of $\dbN$. 
Write $P=[N^{1/4}]$. The quantity $\Ups({\overline{\calA+2\calB}}\cap \calN,
2(\calB\cap \calL);N)$ is bounded above by the number of solutions of the 
equation
$$n_1-n_2=x_1^4+x_2^4-x_3^4-x_4^4,$$
with $n_1,n_2\in \left({\overline{\calA+2\calB}}\cap \calN\right)_{2N}^{3N}$ and 
$x_i\in (\calT)_0^{P}$ $(1\le i\le 4)$. Putting $\calK=\calB$ and $\calZ=
{\overline{\calA+2\calB}}\cap \calN$, then in the notation associated with the 
statement of Lemma \ref{lemma5.1}, we obtain the bound
$$\Ups({\overline{\calA+2\calB}}\cap \calN,2(\calB \cap \calL);N)\le \int_0^1
|g(\alp)^4K(\alp)^2|\,d\alp.$$
The estimate supplied by Lemma \ref{lemma5.1} therefore yields the relation
\begin{align*}
\Ups({\overline{\calA+2\calB}}\cap \calN,2(\calB \cap \calL);N)\ll &\,N^{3/4}
\left|{\overline{\calA+2\calB}}\cap\calN\right|_{2N}^{3N}\\
&\, +N^{1/2+\eps}\left(\left|{\overline{\calA+2\calB}}\cap\calN\right|_{2N}^{3N}
\right)^{3/2}.
\end{align*}
We substitute this estimate into the conclusion of Theorem \ref{theorem2.1}, and
 thereby deduce that
\begin{align*}
\left(\langle 2\calB\wedge 2\calL\rangle_0^N\right)^2\left|{\overline{
\calA+2\calB}}\cap \calN\right|_{2N}^{3N}&\\
\ll N^{3/4}\left|\calAbar\cap \calM\right|_N^{3N}&+N^{1/2+\eps}\left|\calAbar\cap 
\calM\right|_N^{3N}\left(\left|{\overline{\calA+2\calB}}\cap \calN\right|_{2N}^{3N}
\right)^{1/2}.
\end{align*}

\par The number of representations of a positive integer $n$ as the sum of two 
integral squares is at most $O(n^\eps)$ (this result is classical). It follows 
that the number of representations of a positive integer $n$ as the sum of two 
biquadrates is also $O(n^\eps)$. Since $\calB$ is a high-density subset of the 
biquadrates relative to $\calL$, we deduce that $\langle 2\calB \wedge 2\calL
\rangle_0^N\gg N^{1/2-\eta}$ for every positive number $\eta$. Thus we arrive at 
the upper bound
\begin{align*}
\left|{\overline{\calA+2\calB}}\cap \calN\right|_{2N}^{3N}\ll &\, N^{\eps-1/4}\left|
\calAbar\cap \calM\right|_N^{3N}\\
&\, +N^{\eps-1/2}\left| \calAbar \cap\calM\right|_N^{3N}
\left(\left|{\overline{\calA+2\calB}}\cap \calN \right|_{2N}^{3N}\right)^{1/2},
\end{align*}
whence
$$\left|{\overline{\calA+2\calB}}\cap \calN\right|_{2N}^{3N}\ll N^{\eps-1/4}\left|
\calAbar\cap \calM\right|_N^{3N}+\left( N^{\eps-1/2}\left| \calAbar \cap\calM
\right|_N^{3N}\right)^2,$$
and the conclusion of the theorem follows.
\end{proof}

\begin{proof}[The proof of Theorem \ref{theorem1.6}]
When $x\in \dbN$, one has $x^4\equiv 0$ or $1\pmod{16}$. Put $\calB=\{x^4:x\in 
\dbN\}$ and $\calL=\{n\in \dbN:\text{$n\equiv 0$ or $1$ modulo $16$}\}$. Also, 
when $s$ is a natural number, write
$$\calN_s=\{ n\in \dbN: \text{$n\equiv r\mmod{16}$ with $1\le r\le s$}\}.$$
Then $\calL$ and $\calN_s$ $(s\ge 1)$ are unions of arithmetic progressions 
modulo $16$ satisfying the condition that $\calN_{s+2}=\calN_s+2\calL$.

\par Observe next that, when $s\ge 7$ and $1\le r\le s$, a classical 
application of Bessel's inequality leads from the argument underlying the proof 
of Theorem 1.2 of \cite{Vau1989}, via Theorem 2 of \cite{BW2000}, to the 
estimate
\begin{equation}\label{6.w3}
\left| {\overline{s\calB}}\cap \calN_s\right|_N^{3N}\le Y_s(3N)\ll 
N^{1-(s-6)/16-\del_1},
\end{equation}
in which $\del_1>0.00914$. Here, we have implicitly applied Weyl's inequality 
for superfluous variables in the familiar manner. Although we will not go into 
details within this paper, the argument required in order to treat the major 
arcs, in the implicit application of the circle method, follows along the lines 
of that described in \S3 of \cite{slim4}. The key ingredient is Lemma 5.4 of 
\cite{VW2000}, which allows for the successful analysis of a sixth moment 
involving four smooth and two classical biquadratic Weyl sums. This completes 
our sketch of the proof of the estimates recorded in Theorem \ref{theorem1.6} 
for $Y_s(N)$ when $7\le s\le 9$.\par

We now provide an estimate for $Y_s(N)$ when $s=10$ and $11$. The set $\calB$ is
 trivially a high-density subset of the biquadrates relative to $\calL$, and so 
it follows from Theorem \ref{theorem6.2w} that
\begin{align*}
\left|{\overline{8\calB+2\calB}}\cap \calN_{10}\right|_{2N}^{3N}&\ll N^{\eps-1/4}
\left|{\overline{8\calB}}\cap \calN_8\right|_N^{3N}+N^{\eps-1}\left(\left|
{\overline{8\calB}}\cap\calN_8\right|_N^{3N}\right)^2\\
&\ll N^{\eps-1/4}Y_8(3N)+N^{\eps-1}\left( Y_8(3N)\right)^2.
\end{align*}
Then from (\ref{6.w3}) we see that
$$\left|{\overline{10\calB}}\cap \calN_{10}\right|_{2N}^{3N}\ll N^{5/8-\del_1+\eps}+
N^{3/4-2\del_1+\eps}.$$
Thus, again making use of the notation from (\ref{4.w1}), and with 
$\del=0.00914$, one deduces that
$$Y_{10}(N)\le 3+\sum_{j=1}^J\left|{\overline{10\calB}}\cap \calN_{10}
\right|_{2N_j}^{3N_j}\ll N^{3/4-2\del}.$$
Meanwhile, when $s=11$, in like manner Theorem \ref{theorem6.2w} delivers the 
bound
$$\left|{\overline{9\calB+2\calB}}\cap \calN_{11}\right|_{2N}^{3N}\ll N^{\eps-1/4}
Y_9(3N)+N^{\eps-1}(Y_9(3N))^2.$$
Then from (\ref{6.w3}) we see that
$$\left|{\overline{11\calB}}\cap \calN_{11}\right|_{2N}^{3N}\ll N^{9/16-\del_1+\eps}+
N^{5/8-2\del_1+\eps}.$$
Thus, again making use of the notation from (\ref{4.w1}), and with 
$\del=0.00914$, one deduces that
$$Y_{11}(N)\le 3+\sum_{j=1}^J\left|{\overline{11\calB}}\cap \calN_{11}
\right|_{2N_j}^{3N_j}\ll N^{5/8-2\del}.$$
This completes the proof of Theorem \ref{theorem1.6}. 
\end{proof}

\section{Further remarks on abstract exceptional sets} Since the formulation of 
exceptional sets underlying our statement of Theorem \ref{theorem1.1} would 
appear to be novel to the literature, it seems worthwhile to explore some 
alternative approaches and associated consequences.\par

We begin by providing a formulation of Theorem \ref{theorem1.1} which, though 
equivalent, is sometimes more transparent in its application. In this context, 
when $\calC$ and $\calD$ are subsets of $\dbN$, it is convenient to define 
$\Upshat(\calC,\calD;N)$ to be the number of solutions of the equation 
(\ref{1.1}) with $c_1,c_2\in (\calC)_{2N}^{3N}$, $d_1,d_2\in (\calD)_0^N$ and 
$c_1\ne c_2$. We then have
$$\Ups(\calC,\calD;N)=\left| \calC\right|_{2N}^{3N}\left| \calD\right|_0^N+
\Upshat(\calC,\calD;N).$$
Let $N$ be a large natural number, and suppose that $\calA,\calB\subseteq \dbN$.
 Then, when
$$\Upshat({\overline{\calA+\calB}},\calB;N)\le  |\calB|_0^N\left| 
{\overline{\calA+\calB}}\right|_{2N}^{3N},$$
it follows from Theorem \ref{theorem1.1} that
$$\left( |\calB|_0^N\left| {\overline{\calA+\calB}}\right|_{2N}^{3N}\right)^2\le 
2\left| \calAbar\right|_N^{3N}|\calB|_0^N\left| {\overline{\calA+\calB}}
\right|_{2N}^{3N}.$$
Meanwhile, when
$$\Upshat({\overline{\calA+\calB}},\calB;N)>|\calB|_0^N\left| {\overline{
\calA+\calB}}\right|_{2N}^{3N},$$
then instead Theorem \ref{theorem1.1} yields
$$\left( |\calB|_0^N\left| {\overline{\calA+\calB}}\right|_{2N}^{3N}\right)^2<
2\left| \calAbar\right|_N^{3N}\Upshat({\overline{\calA+\calB}},\calB;N).$$
We summarise this interpretation in the following theorem.

\begin{theorem}\label{theorem6.1}
Suppose that $\calA,\calB\subseteq \dbN$. Then for each natural number $N$, one 
has either
$$|\calB|_0^N\left| {\overline{\calA+\calB}}\right|_{2N}^{3N}\le 2\left| \calAbar
\right|_N^{3N},$$
or else
$$\left( |\calB|_0^N\left| {\overline{\calA+\calB}}\right|_{2N}^{3N}\right)^2\le 
2\left| \calAbar\right|_N^{3N}\Upshat({\overline{\calA+\calB}},\calB;N).$$
\end{theorem}

An immediate application of Theorem \ref{theorem6.1} relates to exceptional set 
problems involving prime numbers. Let $\calA$ be a subset of $\dbN$, and suppose
 that $\calP$ is a non-empty subset of the prime numbers. For a fixed non-zero 
value of the integer $n_1-n_2$, standard sieve methods show that there is a 
positive number $C$ with the property that the number of solutions of the 
equation 
$p_1-p_2=n_1-n_2$, in prime numbers $p_1,p_2\in (\calP)_0^N$, is at most 
$$C\frac{N}{(\log N)^2}\prod_{\substack{{\text{$p$ prime}}\\ p|(n_1-n_2)}}
\Bigl( \frac{p-1}{p-2}\Bigr)\ll \frac{N\log \log N}{(\log N)^2}.$$
Thus we find that
$$\Upshat({\overline{\calA+\calP}},\calP;N)\ll \left( \left| {\overline{
\calA+\calP}}\right|_{2N}^{3N}\right)^2N\frac{\log \log N}{(\log N)^2}.
$$
The second alternative of Theorem \ref{theorem6.1} therefore implies that
$$\left| {\overline{\calA+\calP}}\right|_{2N}^{3N}\ll \left( \frac{\left| 
\calAbar\right|_N^{3N}}{N(\log \log N)^{-1}}\right)^{1/2}\left( \frac{N(\log N)^{-1}
}{\left|\calP\right|_0^N}\right) \left| {\overline{\calA+\calP}}\right|_{2N}^{3N}.
$$
In situations wherein $\left|\calAbar\right|_N^{3N}=o(N/\log \log N)$ and 
$\calP$ has positive Dirichlet density, this yields the estimate
$$\left| {\overline{\calA+\calP}}\right|_{2N}^{3N}=o\left( \left| 
{\overline{\calA+\calP}}\right|_{2N}^{3N}\right).$$
Thus one finds that $\left| {\overline{\calA+\calP}}\right|_{2N}^{3N}=0$ for large
 enough values of $N$, a conclusion that plainly holds in much wider generality 
than this illustrative example suggests. Meanwhile, the first alternative of 
Theorem \ref{theorem6.1} yields the bound
$$\left| {\overline{\calA+\calP}}\right|_{2N}^{3N}\le 2\left| \calAbar
\right|_N^{3N}/\left| \calP\right|_0^N\ll \frac{\log N}{N}\left|\calAbar
\right|_N^{3N}.$$
In particular, in the scenario under consideration, one finds that 
$$\left| {\overline{\calA+\calP}}\right|_{2N}^{3N}=o\left(\frac{\log N}
{\log \log N}\right),$$
an exceptionally slim exceptional set estimate.\par

Consider next problems wherein the set $\calB$ is a non-empty set of natural 
numbers supported on the values of a polynomial sequence of degree exceeding 
one. A divisor function estimate shows that
$$\Upshat({\overline{\calA+\calB}},\calB;N)\ll N^\eps \left( \left| 
{\overline{\calA+\calB}}\right|_{2N}^{3N}\right)^2.$$
In such a situation, the first alternative of Theorem \ref{theorem6.1} supplies 
the bound
\begin{equation}\label{7.new0}
\left| {\overline{\calA+\calB}}\right|_{2N}^{3N}\le 2\left| \calAbar
\right|_N^{3N}/|\calB|_0^N.
\end{equation}
In the alternative case, one finds that
$$\left| {\overline{\calA+\calB}}\right|_{2N}^{3N}\ll N^\eps \left( \left| \calAbar
\right|_N^{3N}\left( |\calB|_0^N\right)^{-2}\right)^{1/2}\left| {\overline{\calA
+\calB}}\right|_{2N}^{3N}.$$
Thus, provided that
\begin{equation}\label{7.new1}
|\calB|_0^N>N^\del \left( \left| \calAbar\right|_N^{3N}\right)^{1/2},
\end{equation}
for some fixed positive number $\del$, and all large values of $N$, then we 
deduce that $\left| {\overline{\calA+\calB}}\right|_{2N}^{3N}\ll N^{-\del/2}\left| 
{\overline{\calA+\calB}}\right|_{2N}^{3N}$. The latter implies that $\left| 
{\overline{\calA+\calB}}\right|_{2N}^{3N}=0$ for large enough values of $N$. In 
either case, therefore, provided that the condition (\ref{7.new1}) holds, then 
for large values of $N$ one has the upper bound (\ref{7.new0}).\par

It is natural to speculate concerning the true magnitude of the improvement in 
the exceptional set estimates available from the addition of a set $\calB$. An 
example at one end of the spectrum is given by taking a set $\calA\subseteq 
\dbN$ with the property that $\calAbar$ is supported on even numbers only, and 
$\calB=\{ 0,1\}$. Then $\calA+\calB=\dbN$, so that no matter what the 
cardinality of $\left(\calAbar\right)_N^{3N}$ may be, one has $\left| 
{\overline{\calA+\calB}}\right|_{2N}^{3N}=0$ with $|\calB|_0^N=2$. In addition, if
 $\calA\subseteq \dbN$ has the property that $|\calA|_N^{3N}=o(N)$ as $N$ tends 
to infinity, then $\left| \calAbar \right|_N^{3N}\gg N$, and a priori there is no
 reason to suppose that $\left| {\overline{\calA+\calB}}\right|_{2N}^{3N}$ is 
$o(N)$.\par

The typical situation is probably reflected by a heuristic argument based on 
the application of the Hardy-Littlewood method. On estimating the contribution 
anticipated from the major arcs, one is led to the following speculation.

\begin{conjecture}\label{conjecture6.2}
Suppose that $\calA,\calB\subseteq \dbN$. Then one has
\begin{equation}\label{6.4}
\Upshat({\overline{\calA+\calB}},\calB;N)\ll N^{\eps-1}\left( \left| {\overline{
\calA+\calB}}\right|_{2N}^{3N}|\calB|_0^N\right)^2.
\end{equation}
\end{conjecture}

If we substitute the conjectured estimate (\ref{6.4}) into the second 
alternative of Theorem \ref{theorem6.1}, then one obtains the estimate
$$\left| {\overline{\calA+\calB}}\right|_{2N}^{3N}\ll N^{\eps-1}\left| {\overline{
\calA+\calB}}\right|_{2N}^{3N}\left| \calAbar\right|_N^{3N}.$$
When $\calA$ has complementary density growth exponent smaller than $1$, 
therefore, it follows that for some positive number $\del$, one has
$$\left| {\overline{\calA+\calB}}\right|_{2N}^{3N}\ll N^{\eps-\del}\left| 
{\overline{\calA+\calB}}\right|_{2N}^{3N},$$
which for sufficiently large values of $N$ implies that $\left| {\overline{
\calA+\calB}}\right|_{2N}^{3N}=0$. The first alternative of Theorem 
\ref{theorem6.1}, meanwhile, yields the bound
$$\left| {\overline{\calA+\calB}}\right|_{2N}^{3N}\le 2\left| \calAbar 
\right|_N^{3N}/|\calB|_0^N.$$
We may therefore conclude as follows.

\begin{corollary}\label{corollary6.3}
Suppose that $\calA$ and $\calB$ are non-empty subsets of $\dbN$, and that 
$\calA$ has complementary density growth exponent smaller than $1$. Assume the 
validity of Conjecture \ref{conjecture6.2}. Then for all large values of $N$, 
one has
$$\left| {\overline{\calA+\calB}}\right|_{2N}^{3N}\le 2\left| \calAbar 
\right|_{N}^{3N}/|\calB|_0^N.$$
\end{corollary}

The validity of the conditional estimate of Corollary \ref{corollary6.3} would 
have far reaching consequences. Let $\calX=\{ n^k:n\in \dbN\}$. As usual, we 
define $G(k)$ to be the least natural number $s$ with the property that all 
sufficiently large integers are the sum of at most $s$ $k$th powers of natural 
numbers. Equivalently, the number $G(k)$ is the least natural number $s$ for 
which $|{\overline{s\calX}}|_0^N=O(1)$ for all large $N$. Also, let $G_1(k)$ 
denote the least natural number $s_1$ for which $s_1\calX$ has complementary 
density growth exponent smaller than $1$. Thus, when $s\ge G_1(k)$, almost all 
natural numbers are the sum of at most $s$ $k$th powers of natural numbers.\par

By repeated application of the conditional Corollary \ref{corollary6.3}, one 
deduces that for $G_1(k)<t<G_1(k)+k$, one has
$$\left| {\overline{t\calX}}\right|_{2N}^{3N}\ll \left| {\overline{(t-1)\calX}}
\right|_N^{3N}N^{-1/k},$$
whence $t\calX$ has complementary density growth exponent smaller than
$$1-(t-G_1(k))/k.$$
In this way, one finds that
$$G(k)\le G_1(k)+k.$$
The methods of Wooley \cite{Woo1992}, \cite{Woo1995}, in combination with a 
classical application of Bessel's inequality, yield the estimate
$$G_1(k)\le {\textstyle{\frac{1}{2}}}k(\log k+\log \log k+2+o(1)),$$
from which we deduce the conditional upper bound
$$G(k)\le {\textstyle{\frac{1}{2}}}k(\log k+\log \log k+4+o(1)).$$
Of course, it seems likely that theoretical advances sufficient to establish 
Conjecture \ref{conjecture6.2} would already yield an estimate of the shape 
$G(k)=O(k)$.

\bibliographystyle{amsbracket}
\providecommand{\bysame}{\leavevmode\hbox to3em{\hrulefill}\thinspace}

\end{document}